\documentclass[11pt,reqno]{amsart}
\usepackage{longtable}
\usepackage{hyperref}
\usepackage[T1]{fontenc}
\usepackage[utf8]{inputenc}
\usepackage[english]{babel}
\usepackage{textcomp}
\usepackage{dsfont}
\usepackage{latexsym}
\usepackage{amssymb}
\usepackage{amsthm}
\usepackage{amsmath}
\DeclareMathAlphabet{\mathpzc}{OT1}{pzc}{m}{en}
\usepackage{yfonts}
\usepackage{xfrac}
\usepackage{newlfont}
\usepackage{graphicx}
\usepackage{mathtools}
\usepackage{comment}
\usepackage{indentfirst}
\usepackage{braket}
\usepackage{mathrsfs}
\usepackage{fixmath}

\usepackage{etoolbox}
\apptocmd{\lim}{\limits}{}{}
\apptocmd{\sup}{\limits}{}{}
\apptocmd{\inf}{\limits}{}{}
\apptocmd{\liminf}{\limits}{}{}
\apptocmd{\limsup}{\limits}{}{}
\pretocmd{\langle}{\left}{}{}
\pretocmd{\rangle}{\right}{}{}

\usepackage{scalerel}[2014/03/10]
\usepackage[usestackEOL]{stackengine}
\newcommand{\dashint}{\,\ThisStyle{\ensurestackMath{%
			\stackinset{c}{.2\LMpt}{c}{.5\LMpt}{\SavedStyle-}{\SavedStyle\phantom{\int}}}%
		\setbox0=\hbox{$\SavedStyle\int\,$}\kern-\wd0}\int}

\DeclareMathOperator{\supp}{Supp}

\newcommand{\dd}{\mathrm{d}}

\DeclarePairedDelimiter{\abs}{\lvert}{\rvert}
\DeclarePairedDelimiter{\norm}{\lVert}{\rVert}

\let\originalleft\left
\let\originalright\right
\renewcommand{\left}{\mathopen{}\mathclose\bgroup\originalleft}
\renewcommand{\right}{\aftergroup\egroup\originalright}

\newcommand{\N}{\mathds{N}}

\newcommand{\R}{\mathds{R}}

\newcommand{\Mf}{\mathfrak{M}}

\newcommand{\Gc}{\mathcal{G}}
\newcommand{\Hc}{\mathcal{H}}

\newcommand{\meg}{\leqslant}
\newcommand{\Meg}{\geqslant}
\newcommand{\eps}{\varepsilon}
\renewcommand{\phi}{\varphi}
\newcommand{\mi}{\mu}

\begin{document}

\theoremstyle{definition}
\newtheorem{deff}{Definition}

\newtheorem{oss}[deff]{Remark}

\theoremstyle{plain}
\newtheorem{teo}[deff]{Theorem}

\newtheorem{lem}[deff]{Lemma}

\newtheorem{prop}[deff]{Proposition}

\newtheorem{cor}[deff]{Corollary}

\author[M. Calzi]{Mattia Calzi}

\title{An Integral Version of   Hardy's Inequality}

\address{Dipartimento di Matematica, Universit\`a degli Studi di
	Milano, Via C. Saldini 50, 20133 Milano, Italy}
\email{{\tt mattia.calzi@unimi.it}}

\thanks{{\em Math Subject Classification 2020} }
\thanks{The author is a member of the 	Gruppo Nazionale per l'Analisi Matematica, la Probabilit\`a e le	loro Applicazioni (GNAMPA) of the Istituto Nazionale di Alta	Matematica (INdAM). The author was partially funded by the INdAM-GNAMPA Project CUP\_E55F22000270001.
} 

\begin{abstract}
	In this note we present a version of Hardy's inequality on a measure space $(X,\mi)$ endowed with  a measurable function $N\colon X\to \R$ which replaces the absolute value on $\R$ or $\R^n$, and, more generally, the distance function from a given point when $X$ is a metric space.
\end{abstract}
\maketitle

\section{Introduction}

In~\cite{Hardy2,Landau}, Hardy's inequality was first proved in its discrete version 
\[
\sum_{n=1}^\infty \left( \frac 1 n \sum_{k=1}^n a_k \right)^p\meg \left( \frac{p}{p-1} \right)^p \sum_{n=1}^\infty a_n^p
\]
for every $p\in (1,\infty)$ and for every positive sequence $(a_n)$. The continuous version
\[
\int_0^{+\infty} \left(\frac 1 x \int_0^x f(t)\,\dd t  \right)^p\,\dd x\meg  \left( \frac{p}{p-1} \right)^p  \int_0^{+\infty} f(x)^p\,\dd x,
\]
for every positive measurable function $f$ on $[0,\infty)$ and for every $p>1$, was later proved in~\cite{Hardy1}. Since then, Hardy's inequality has been generalized and extended to several contexts. On the one hand, the function $1/x$ has been replaced with more general weights in a series of works (cf., e.g.,~\cite{Talenti,Tomaselli,Muckenhoupt}), finally reaching a characterization of the positive Radon measures $\mi,\nu$ for which  there is a $C\Meg 0$ such that
\begin{equation}\label{eq:2}
\int_0^{+\infty} \left( \int_0^x f(t)\,\dd t  \right)^p\,\dd \mi( x)\meg C^p \int_0^{+\infty} f(x)^p\,\dd \nu(x)
\end{equation}
for every positive Borel measurable function $f$ on $\R$. Explicitly, there is a finite constant $C\Meg 0$ as above if and only if
\[
B\coloneqq \sup_{r>0} \mi([r,+\infty))^{1/p}\left( \int_0^r \psi^{-p'/p}(x)\,\dd x\right)^{1/p'}<\infty
\]
(interpreting $\big( \int_0^r \psi^{-p'/p}(x)\,\dd x\big)^{1/p'}$ as $\norm{\chi_{(0,r)}\psi^{-1/p}}_{L^{p'}(\R)}$ when $p=1$), where $\psi$ is the density of the absolutely continuous part of $\nu$ with respect to Lebesgue measure, in which case the minimal constant $C$ such that~\eqref{eq:2} holds satisfies $B\meg C\meg (p')^{1/p'} p^{1/p} B$, with the conventions $0^0=0$ and $\infty^0=1$. A similar characterization was also provided for $p=\infty$, with suitable modifications. 

On the other hand, $\R$ has been replaced with more general spaces, from $\R^n$, replacing the segments $[0,r]$ with the balls $B(0,r)$, to general metric spaces, with similar substitutions, cf.~\cite{RV,RV2}. These latter papers provide a very general version of Hardy's inequality on metric spaces based on a suitable `polar decomposition'. On the contrary, we shall use a duality method in order to essentially extend the results of~\cite{Muckenhoupt}. For more detailed accounts on the available literature on Hardy's inequalities (including several equivalent formulations of the finiteness of $B$, in the above notation), cf., e.g.,~\cite{AM,OK,Davies,PS,GKPW,EE,GM,GKP,Mazya,BEL,KPS,RV,RV2} and the references therein.

In this note, we shall prove the following general version of Hardy's inequality.

\begin{teo}\label{teo:1}
	Let $(X,\Mf)$ be a measurable space, let $\mi_2$ be a positive Borel measure on $\R$, and let $\mi_1,\mi_3$ be two  positive measures on $(X,\Mf)$, with $\mi_1$ semi-finite. Take $p,q\in [1,\infty]$ with $p\meg q$, and assume that the absolutely continuous part of $\mi_3$ with respect to $\mi_1$ has a density $\psi\colon X\to [0,+\infty)$ with respecto to $\mi_1$.\footnote{This is certainly the case if $\mi_1$ and $\mi_3$ are both Radon or $\sigma$-finite. One may simply require $\mi_1$ to be Radon or $\sigma$-finite if one allows $\psi$ to take also the value $+\infty$, but in this case the assertion is still equivalent to Theorem~\ref{teo:1}, applied to $X\setminus \psi^{-1}(+\infty)$.} In addition, let $N\colon X\to \R$ be a $\Mf$-measurable function. Then, there is a constant $C>0$ such that
	\begin{equation}\label{eq:1}
		\left( \int_\R \left( \int_{N(x)<r} f(x)\,\dd \mi_1(x)\right) ^q \,\dd \mi_2(r)\right)^{1/q}\meg C  \left( \int_X f(x)^p\,\dd \mi_3(x)\right)^{1/p}
	\end{equation}
	(with the usual modifications when $p$ or $q$ is $\infty$), for every positive $\Mf$-measurable function $f$ on $X$,
	if and only if
	\[
	B\coloneqq \sup_{r\in\R}\mi_2([r,+\infty))^{1/q} \left(\int_{N(x)<r} \psi^{-p'/p}\,\dd \mi_1\right)^{1/p'}
	\]
	(with the usual modification when $p=1$ and the conventions $0^0=0$ and $\infty^0=1$) is finite.
	The minimal $C$ for which~\eqref{eq:1} holds satisfies the inequalities $B\meg C\meg \min((p')^{1/p'} p^{1/q}, (q')^{1/p'} q^{1/q} )B$.
\end{teo}

Some comments are in order. First of all, we say that a  measure $\mi$ on $(X,\Mf)$  is semi-finite if every $\Mf$-measurable subset $E$ of $X$ with $\mi(E)=+\infty$ contains a $\Mf$-measurable subset $F$ with $0<\mi(F)<+\infty$. This is readily seen to be equivalent to the fact that $\mi(E)$ is the least upper bound of the $\mi(F)$, as $F$ runs through the set of $\Mf$-measurable subsets of $E$ with finite measure, for every $\Mf$-measurable subset $E$ of $X$. Thus, $\mi$ is semi-finite if it satisfies some kind of inner regularity with respect to $\Mf$-measurable sets with finite measure. In particular, bounded and $\sigma$-finite measures are semi-finite.

Furthermore, given a positive $\Mf$-measurable function $\psi$, we denote by $\psi\cdot \mi$ the measure such that $(\psi\cdot \mi)(E)=\int_E \psi\,\dd \mi$ for every $\Mf$-measurable subset $E$ of $X$, in which case we say that $\psi$ is a density of $\psi\cdot \mi$ with respect to $\mi$. We remark that, when we say `positive', we mean `$\Meg0$', so that a positive function is a function which is $\Meg 0$ everywhere.

Let us now briefly comment on the assumptions of Theorem~\ref{teo:1}.
Observe first that, if $\mi_1$ is not assumed to be semi-finite, then the inequality $B\meg C$ may fail (and the existence of $C$ may fail to imply the finiteness of $B$). For example, if $\mi_1(E)=+\infty$ for every non-empty element of $\Mf$, and if $\mi_3=\mi_1$ (so that one may take $\psi=1$), then~\eqref{eq:1} holds trivially for every positive $\Mf$-measurable  function $f$ on $X$ with any non-zero $C$, but $B$ is infinite unless $\mi_2=0$ on $(\inf N(X),+\infty) $.

Furthermore, notice that
\begin{equation}\label{eq:3}
 \sup_{r\in\R}\mi_2([r,+\infty))^{1/q} \left(\int_{N(x)<r} \psi^{-p'/p}\,\dd \mi_1\right)^{1/p'}=  \sup_{r\in\R}\mi_2((r,+\infty))^{1/q} \left(\int_{N(x)\meg r} \psi^{-p'/p}\,\dd \mi_1\right)^{1/p'}
\end{equation}
since the two sides are the least upper bound of a left-continuous and a right-continuous function which coincide  on the complement of a countable set (if $q<\infty$ and $p>1$; the other cases require a slightly different treatment). On the other hand,
\[
B'\coloneqq  \sup_{r\in\R}\mi_2([r,+\infty))^{1/q} \left(\int_{N(x)\meg r} \psi^{-p'/p}\,\dd \mi_1\right)^{1/p'}
\]
may be strictly larger than $B$ (and is crucial in Theorem~\ref{teo:1bis}), while
\[
\sup_{r\in\R}\mi_2((r,+\infty))^{1/q} \left(\int_{N(x)< r} \psi^{-p'/p}\,\dd \mi_1\right)^{1/p'}
\]
may be strictly smaller than $B$.

Finally, observe that, if there is a  measure $\nu$ on $(X,\Mf)$ such that   $N_*(\nu)=\mi_2$, that is, $\mi_2(A)=\nu(N^{-1}(A))$ for every Borel subset $A$ of $\R$, then~\eqref{eq:1} may be written in the equivalent way
\[
\left( \int_X \left( \int_{N(x)<N(y)} f(x)\,\dd \mi_1(x)\right) ^q \,\dd \nu(y)\right)^{1/q}\meg C  \left( \int_X f(x)^p\,\dd \mi_3(x)\right)^{1/p}
\]
for every positive $\Mf$-measurable function $f$ on $X$, since the function $r \mapsto \int_{N(x)<r} f(x)\,\dd \mi_1(x)$ is increasing, hence Borel measurable. Nonetheless, even if $N(X)=\R$, there may be positive Borel measures on $\R$ which are not images of a positive  measure on $(X,\Mf)$\footnote{For example, if $X=\R$ and $\Mf$ is the set of subsets of $X$, then   $N_*(\nu)$ is the restriction of $\nu$ to  the Borel $\sigma$-algebra of $\R$ and it seems unlikely that every Borel measure  $\mi_2$ on $\R$ may be written in the form $N_*(\nu)$ for some  measure $\nu$ on $(X,\Mf)$. Under the continuum hypothesis, $\R$ is ulamian, so that every $\nu$ as above must be atomic, that is, $\nu=\sum_{x\in \R} f(x) \delta_x$ for some $f\colon X\to [0,+\infty]$, so that only the atomic measures on $\R$ may be of the form $N_*(\nu)$.}

Let us now compare Theorem~\ref{teo:1} with some results   in the literature.
First of all, the discrete version of Hardy's inequality may be obtained from Theorem~\ref{teo:1} choosing $X=1+\N=\Set{1,2,3,\dots}$, $\mi_1=\mi_3$ as the counting measure on $X$, $\mi_2=\sum_{n=1}^\infty \frac{1}{n^p} \delta_{n+1}$, and letting $N$ be the natural inclusion of $X$ in $\R$ (nonetheless, the discrete version of Hardy's inequality is most naturally obtained as a particular case of Theorem~\ref{teo:1bis} below). Analogously, the continuous version of Hardy's inequaliy may be  obtained choosing $X=(0,\infty)$, $\mi_1=\mi_3$ as Lebesgue measure, $\mi_2=\chi_{(0,+\infty)}\frac{1}{(\,\cdot\,)^p}\cdot \Hc^1$, where $\Hc^1$ denotes the $1$-dimensional Hausdorff measure, that is, Lebesgue measure, and letting $N$ be the natural inclusion  of $X$ in $\R$. The general version~\eqref{eq:2} of Hardy's inequality on $[0,\infty)$ may be obtained in a similar way, as well as several other versions of Hardy's inequality on $\R^n$ on in more general metric spaces (including~\cite{RV}). 

Furthermore, other interesting examples may be studied by means of Theorem~\ref{teo:1}. For instance, Theorem~\ref{teo:1} may be applied to every Riemannian manifold (complete or not), endowed with the canonical volume form, and even to every sub-Riemannian manifold, endowed with a suitable  measure (such as Popp's measure, when defined). Once the behaviour of the volume of the balls centred at a fixed point near $0$ and near $\infty$ is understood, one may find estimates for $B$ and then sufficient conditions for~\eqref{eq:1} to hold with a certain $C$.
Besides that, Thereom~\ref{teo:1} applies also to discrete spaces, as remarked earlier, even though Theorem~\ref{teo:1bis} below applies more naturally, in this latter context.

In addition, the so-called conjugate (or dual) Hardy's inequality  (which we now express on $[0,\infty)$ for the sake of simplicity)
\[
\int_0^{+\infty} \left( \int_x^{+\infty} f(t)\,\dd t  \right)^p\,\dd \mi( x)\meg C^p \int_0^{+\infty} f(x)^p\,\dd \nu(x)
\]
may be obtained as a particular case of~\eqref{eq:1} choosing $N(x)=-x$.

Notice that we do \emph{not} consider the cases in which $\min(p,q)<1$ or $p<q$ (cf., e.g.,~\cite{OK,PS,RV2}). Indeed, the former cannot be studied with duality methods; the latter is not considered for the sake of simplicity.

Concerning the vast wealth of equivalent formulations of the finiteness of $B$ (cf., e.g.,~\cite{GKPW,GKP}), we observe that, at least when $\mi_2$ and $N_*(\mi_1)$ are $\sigma$-finite and absolutely continuous  with respect to Lebesgue measure, no new proofs are necessary, as these equivalent conditions are formulated in terms of the left-continuous monotone functions $r\mapsto\mi_2([r,+\infty))$ and $r\mapsto \int_{N(x)<r} \psi^{-p'/p}(x)\,\dd \mi_1(x)$ and their derivatives $-\mi_2$ and $N_*(\psi^{-p'/p}\cdot \mi_1)$ (for $p>1$). For more general measures, suitable extensions may be in order. Nonetheless, once these equivalences are formulated for general pairs of left-continuous decresing and increasing functions, they immediately apply to this more general context.

We shall also prove the following variant of Theorem~\ref{teo:1}.

\begin{teo}\label{teo:1bis}
	Keep the hypotheses and the notation of Theorem~\ref{teo:1}. Then, there is a constant $C'>0$ such that
	\begin{equation}\label{eq:1bis}
		\left( \int_\R \left( \int_{N(x)\meg r} f(x)\,\dd \mi_1(x)\right) ^q \,\dd \mi_2(r)\right)^{1/q}\meg C  \left( \int_X f(x)^p\,\dd \mi_3(x)\right)^{1/p}
	\end{equation}
	(with the usual modifications when $p$ or $q$ is $\infty$), for every positive $\Mf$-measurable function $f$ on $X$,
	if and only if
	\[
	B'\coloneqq \sup_{r\in\R}\mi_2([r,+\infty))^{1/q} \left(\int_{N(x)\meg r} \psi^{-p'/p}\,\dd \mi_1\right)^{1/p'}
	\]
	(with the usual modification when $p=1$) is finite.
	The minimal $C'$ for which~\eqref{eq:1bis} holds satisfies the inequalities $B'\meg C'\meg \min((p')^{1/p'} p^{1/q}, (q')^{1/p'} q^{1/q} )B'$.
\end{teo}

\section{Proof of Theorems~\ref{teo:1} and~\ref{teo:1bis}}

We first need to recall a simple lemma (cf.~\cite[Theorem 3.96]{AFP} for a proof of a much more general version of this result). We restrict our attention to diffuse measures (that is, measures $\mi$ such that $\mi(\Set{x})=0$ for every $x$) since this additional assumption simplifies the statement and   this case will be sufficient for our purposes.

\begin{lem}\label{lem:1}
	Let $f$ be a continuous function on an open subset $I$ of $\R$, and assume that its distributional derivative $f'$ is a (Radon) measure.  Let $J$ be an open subset of $\R$ which contains $f(I)$, and let $\phi$ be a function of class $C^1$ on $J$. Then,
	\[
	(\phi\circ f)'=(\phi'\circ f)\cdot f'. 
	\]
\end{lem}

Recall that, for a function $f$ as in the statement of Lemma~\ref{lem:1}, $f(b)-f(a)=f'([a,b])$ for every $a,b\in I$ such that  $a<b$ and $[a,b]\subseteq I$, so that the conclusion of Lemma~\ref{lem:1} may be rephrased saying that $\phi(f(b))-\phi(f(a))= \int_a^b (\phi'\circ f)\,\dd f'$ for every $a,b\in I$ such that $a<b$ and $[a,b]\subseteq I$. In particular, saying that $f$ is continuous is essentially  equivalent to saying that $f'$ is diffuse.

Notice that~\cite[Theorem 3.96]{AFP} requires $\phi$ to be Lipschitz since the authors want to deduce that the measure $(\phi\circ f)'$ is bounded (assuming $f'$   bounded as well). If we do not require this fact and apply the cited result locally, we obtain Lemma~\ref{lem:1}.

We may now pass to the proof of Theorem~\ref{teo:1}.

\begin{proof}[Proof of Theorem~\ref{teo:1}.]
	Throughout the proof, we shall assume that $p,q\in (1,\infty)$. We shall leave to the reader the simple modifications which are needed to deal with the remaining cases.  We shall set   $C_B= \min((p')^{1/p'} p^{1/q}, (q')^{1/p'} q^{1/q} )$. 
	
	\textsc{Step I.} Assume that: $B$ is finite;   $X=\R$; $\Mf$ is the Borel $\sigma$-algebra of $\R$; $N(r)=r$ for every $r\in \R$; $\mi_1,\mi_2$ are diffuse Radon measures with compact support on $\R$; $\mi_3=\mi_1$ (so that $\psi=1$).  Set $h(r)\coloneqq \mi_1((-\infty,r))^{1/(pp')}$ for every $R\in\R$. Observe that~\eqref{eq:1} holds trivially with $C=0$ if $\mi_1=0$, that is, if $h=0$, so that we may assume that $h\neq 0$. Let $r_1$ be the maximum of the $r$ such that $h(r)=0$.  Analogously, we may assume that $\mi_2\neq 0$, in which case we define $r_2$ as the minimum of the $r$ such that $\mi_2((r,+\infty))=0$. We may assume that $r_1<r_2$, for otherwise the assertion is trivial.
	
	Then, take a positive $\Mf$-measurable function $f$ on $\R$, and observe that, by H\"older's and Minkowski's integral inequality,
	\[
	\begin{split}
		\left(\int_\R \left( \int_{-\infty}^r f\,\dd \mi_1  \right)^q \,\dd \mi_2(r) \right)^{p/q}&\meg \left(\int_{r_1}^{r_2} \left( \int_{r_1}^r (fh)^p\,\dd \mi_1  \right)^{q/p} \left( \int_{r_1}^r h^{-p'}\,\dd \mi_1  \right)^{q/p'}  \,\dd \mi_2(r) \right)^{p/q}\\
		&\meg\int_{r_1}^{r_2} (f h)^p(s) \left( \int_s^{r_2} \left( \int_{r_1}^r h^{-p'}\,\dd \mi_1  \right)^{q/p'}  \,\dd \mi_2(r)\right)^{p/q}\,\dd \mi_1(s). 
	\end{split}
	\]
	Then, observe that, for every $r> r_1$,
	\[
	\int_{r_1}^r h^{-p'}\,\dd \mi_1=\int_{r_1}^r \mi_1((r_1,s))^{-1/p}\,\dd \mi_1(s)= p' \mi_1((r_1,r))^{1/p'}\meg B p' \mi_2([r,+\infty))^{-1/q}
	\]
	by Lemma~\ref{lem:1}. 
	In addition,
	\[
	\begin{split}
		\int_s^{r_2} \left( \int_{r_1}^r h^{-p'}\,\dd \mi_1  \right)^{q/p'}  \,\dd \mi_2(r)&\meg (B p')^{q/p'} \int_s^{r_2}  \mi_2((r,r_2))^{-1/p'}  \,\dd \mi_2(r)\\
		&=(B p')^{q/p'} p \mi_2((s,r_2))^{1/p}
	\end{split}
	\]
	again by Lemma~\ref{lem:1}. Thus,
	\[
	\int_{r_1}^{r_2} (f h)^p(s) \left( \int_s^{r_2} \left( \int_{r_1}^r h^{-p'}\,\dd \mi_1  \right)^{q/p'}  \,\dd \mi_2(r)\right)^{p/q}\,\dd \mi_1(s)\meg (B p')^{p/p'} p^{p/q} B \int_\R f^p\,\dd \mi_1
	\]
	by the definition of $B$. Thus,~\eqref{eq:1} holds for every positive $\Mf$-measurable function $f$ on $X$ with $C= (p')^{1/p'} p^{1/q} B$.
	
	\textsc{Step II.} Assume that: $B$ is finite;  $X=\R$; $\Mf$ is the Borel $\sigma$-algebra of $\R$; $N(r)=r$ for every $r\in\R$; $\mi_1,\mi_2$ are diffuse Radon measures with compact support on $\R$; $\mi_3=\mi_1$. We proceed by duality, using~\textsc{step I}.  Define $\Gc$ as the set of positive $\Mf$-measurable functions $g$ on $\R$ such that $\norm{g}_{L^{q'}(\mi_2)}\meg 1$.
	Take a positive $p$-th power  $\mi_1$-integrable function $f$ on $\R$, and observe that
	\[
	\begin{split}
	\left(\left(\int_{-\infty}^r f\,\dd \mi_1  \right)^q \,\dd \mi_2(r)\right)^{1/q}&= \sup_{g\in \Gc} \int_\R \int_{-\infty}^r f\,\dd \mi_1 	\,g(r)\,\dd \mi_2(r)\\
		&=\sup_{g\in \Gc} \int_{\R} \int_{s}^{+\infty} g\,\dd \mi_2\, f(s)\,\dd \mi_1(s)\\
		&\meg \sup_{g\in \Gc} \left(\int_{\R}\left(\int_{s}^{+\infty} g\,\dd \mi_2\right)^{p'}\,\dd \mi_1(s)\right)^{1/p'}\left(\int_{\R} f^{p}\,\dd \mi_1\right)^{1/p}\\
		&\meg (q')^{1/p'} q^{1/q} B\left(\int_\R f^p\,\dd \mi_1\right)^{1/p},
	\end{split}
	\]
	where the first equality follows from the fact that $\mi_2$ is bounded, the second equality follows from Tonelli's theorem, since both $\mi_1$ and $\mi_2$ are bounded, the first inequality follows from H\"older's inequality, and the last inequality follows from~\textsc{step I}, applied to $\check\mi_2$, $\check \mi_1$, $\check g$, $q'$, and $p'$ instead of $\mi_1$, $\mi_2$, $f$, $p$, and $q$, where $\check\;$ denotes the reflection $r\mapsto -r$ (taking into account~\eqref{eq:3}). Thus,~\eqref{eq:1} holds for every positive $\Mf$-measurable function $f$ on $\R$ with $C=C_B$ (the case in which $f\not \in L^p(\mi_1)$ being trivial).
	
	\textsc{Step III.} Assume that: $B$ is finite; $X=\R$; $\Mf$ is the Borel $\sigma$-algebra of $\R$; $N(r)=r$ for every $r\in\R$; $\mi_1,\mi_2$ are Radon measures with compact support on $\R$, with $\mi_2$ diffuse; $\mi_3=\mi_1$.  Define 
	\[
	\mi_{1,n}\coloneqq  \mi_1+\sum_{r\in \R} ( 2^{n} \chi_{[r,r+2^{-n}]}\cdot \Hc^1-\delta_r)\mi_1(\Set{r})
	\]
	for every $n\in\N$, so that $(\mi_{1,n})$ is a sequence of diffuse positive Radon measures, and
	\[
	\mi_{1,n}((-\infty,r))=\mi_1((-\infty,r))+\sum_{r'<r} \mi_1(\Set{r'}) (\min(1, 2^n (r-r') )-1)  \meg \mi_1((-\infty,r))
	\]
	for every $r\in\R$ and for every $n\in\N$. Let us prove that $\chi_{(-\infty,r)}\cdot \mi_{1,n}$ converges vaguely (that is, in the weak dual of $C_c(\R)$) to $\chi_{(-\infty,r)}\cdot \mi_1$  for every $r\in\R$. Indeed, for every $\phi\in C_c(\R)$,
	\[
	\begin{split}
	\abs*{\langle \chi_{(-\infty,r)}\cdot\mi_1-\chi_{(-\infty,r)}\cdot \mi_{1,n}, \phi\rangle}\meg \sum_{r'<r} \mi_1(\Set{r'}) \abs*{2^n\int_{r'}^{\min(r,r'+2^{-n})} \phi(r'') \,\dd r''-\phi(r') }.
	\end{split}
	\]
	Now, it is clear that $2^n\int_{r'}^{\min(r,r'+2^{-n})} \phi(r'') \,\dd r''$ converges to $\phi(r') $ for every $r'<r$, that 
	\[
	\abs*{2^n\int_{r'}^{\min(r,r'+2^{-n})} \phi(r'') \,\dd r''-\phi(r')  }\meg 2\norm{\phi}_{L^\infty(\R)}
	\] 
	for every $r'<r$, and that $\sum_{r'\in (-\infty,r)} \mi_1(\Set{r})\meg \mi_1(\R)<\infty $. Then, the dominated convergence theorem shows that
	\[
	\lim_{n\to \infty}\abs*{\langle \chi_{(-\infty,r)}\cdot\mi_1-\chi_{(-\infty,r)}\cdot \mi_{1,n}, \phi\rangle}=0,
	\]
	whence our claim by the arbitrariness of $\phi$ and $r$.	 In particular, taking $r\Meg\max \supp \mi_1+1$, we see that $\mi_{1,n}$ converges vaguely to $\mi_1$.
	
	 Then, for every positive $f\in C_c(\R)$, \textsc{steps I} and~\textsc{II} show that
	\[
	\begin{split}
		\left( \int_\R \left( \int_{(-\infty,r)} f\,\dd \mi_1\right) ^q \,\dd \mi_2(r)\right)^{1/q}&\meg \liminf_{n\to \infty }\left( \int_\R \left( \int_{(-\infty,r)} f\,\dd \mi_{1,n}\right) ^q \,\dd \mi_2(r)\right)^{1/q}\\
			&\meg C_B  \liminf_{n\to \infty }\left( \int_\R f^p\,\dd \mi_{1,n}\right)^{1/p}\\
			&=C_B \left( \int_\R f^p\,\dd \mi_{1}\right)^{1/p}
	\end{split}
	\]
	by Fatou's lemma. Thus,~\eqref{eq:1} holds for $f$.	
	Now, if $f\colon \R\to [0,+\infty]$ is  lower semi-continuous, then $f$ is the pointwise limit of an increasing sequence of elements of $C_c(\R)$, so that~\eqref{eq:1} is established for $f$ by monotone convergence. Analogously, if $f$ is a $p$-th power $\mi_1$-integrable positive function, then there is a decresing sequence of positive $p$-th power $\mi_1$-integrable lower semicontinuous functions which converge $\mi_1$-almost everywhere to $f$, so that~\eqref{eq:1} is established for $f$ by monotone convergence and Fatou's lemma. Then~\eqref{eq:1}   holds for every positive $\Mf$-measurable function $f$ on $\R$ (the case in which $f\not \in L^p(\mi_1)$ being trivial).
	
	\textsc{Step IV.}  Assume that: $B$ is finite; $X=\R$; $\Mf$ is the Borel $\sigma$-algebra of $\R$; $N(r)=r$ for every $r\in\R$; $\mi_1,\mi_2$ are Radon measures with compact support on $\R$; $\mi_3=\mi_1$.   Define 
	\[
	\mi_{2,n}\coloneqq \mi_2+\sum_{r\in \R} (2^n \chi_{[r-2^{-n},r]}\cdot \Hc^1-\delta_r)\mi_2(\Set{r}),
	\]
	so that $\mi_{2,n}$ is a sequence of diffuse positive Radon measures such that $\mi_{2,n}([r,+\infty))\meg \mi_2([r,+\infty))$ for every $r\in\R$ and for every $n\in\N$, and such that $\int_\R f\,\dd \mi_{2,n}$ converges to $\int_\R f\,\dd \mi_2$ for every positive \emph{left-continuous} function $f$ on $\R$ (the proof is analogous to that of~\textsc{step III}). Now, observe that, for every positive $p$-th power $\mi_1$-integrable function $f$ on $\R$, the increasing function $r \mapsto \int_{(-\infty,r)} f\,\dd \mi_1$ is left-continuous and bounded (by $\mi_1(\R)^{1/p'} \norm{f}_{L^p(\mi_1)}$). Thus, by means of \textsc{step III} we see that
	\[
	\begin{split}
	\left( \int_\R \left( \int_{(-\infty,r)} f\,\dd \mi_1\right) ^q \,\dd \mi_2(r)\right)^{1/q}&= \lim_{n\to \infty }\left( \int_\R \left( \int_{(-\infty,r)} f\,\dd \mi_{1}\right) ^q \,\dd \mi_{2,n}(r)\right)^{1/q}\\
	&\meg C_B  \left( \int_X f^p\,\dd \mi_{1}\right)^{1/p}.
	\end{split}
	\]
	Thus,~\eqref{eq:1} holds for every positive $\Mf$-measurable function $f$ on $\R$ (the case in which $f\not \in L^p(\mi_1)$ being trivial).
	
	\textsc{Step V.} Assume that: $B$ is finite; $\mi_1,\mi_2$ are bounded; $\mi_1=\mi_3$;  $N_*(\mi_1)$ and $\mi_2$ are compactly supported. Notice that every bounded Borel measure on $\R$ is a Radon measure, so that both $N_*(\mi_1)$ and $\mi_2$ are Radon measures. 
	We proceed by duality as in~\textsc{step II}.  Define $\Gc$ as the set of positive Borel measurable functions $g$ on $\R$ such that $\norm{g}_{L^{q'}(\mi_2)}\meg 1$.
	Take a positive $f\in L^p(\mi_1)$, and observe that
	\[
	\begin{split}
		\left(\left(\int_{N(x)<r} f(x)\,\dd \mi_1(x)  \right)^q \,\dd \mi_2(r)\right)^{1/q}&= \sup_{g\in \Gc} \int_\R \int_{N(x)<r} f(x)\,\dd \mi_1(x) 	\,g(r)\,\dd \mi_2(r)\\
		&=\sup_{g\in \Gc} \int_{X} \int_{N(x)}^{+\infty} g\,\dd \mi_2\, f(x)\,\dd \mi_1(x)\\
		&\meg \sup_{g\in \Gc} \left(\int_{X}\left(\int_{N(x)}^{+\infty} g\,\dd \mi_2\right)^{p'}\,\dd \mi_1(x)\right)^{1/p'}\left(\int_{X} f^{p}\,\dd \mi_1\right)^{1/p}\\
		&=\sup_{g\in \Gc} \left(\int_{\R}\left(\int_{s}^{+\infty} g\,\dd \mi_2\right)^{p'}\,\dd N_*(\mi_1)(s)\right)^{1/p'}\left(\int_{X} f^{p}\,\dd \mi_1\right)^{1/p}\\
		&\meg C \int_\R f^p\,\dd \mi_1,
	\end{split}
	\]
	where each step is justified as in~\textsc{step II}, except for the last equality, which holds since the function $s\mapsto \int_{s}^{+\infty} g\,\dd \mi_2$ is decreasing, hence Borel measurable. Thus,~\eqref{eq:1} holds for every positive $\Mf$-measurable function $f$ on $X$ (the case in which $f\not \in L^p(\mi_1)$ being trivial).

	\textsc{Step VI.} Assume that $B$ is finite. 
	If $ \mi_1(N^{-1}((-\infty,r)))=0$ for some $r\in \R$, let $r_1$ be the maximum of these $r$. Set $r_1\coloneqq -\infty$ otherwise. Analogously, if $\mi_2([r,+\infty))=0$ for some $r\in \R$, let $r_2$ be the greatest lower bound of these $r$. Set $r_2\coloneqq +\infty$ otherwise. 
	Then, the finiteness of $B$ implies that, for every $r\in (r_1,r_2) $, both $\mi_2([r,+\infty))$ and $\int_{N(x)< r} \psi(x)^{-p'/p}\,\dd \mi_1(x)$ are finite and non-zero. Notice that, if $r_2\meg r_1$, then~\eqref{eq:1} trivially holds (with $C=0$) for every positive $\Mf$-measurable function $f$ on $X$, since $\int_{N(y)< r} f(y)\,\dd \mi_1=0$ for every $r\in(-\infty,r_2]$, hence for $\mi_2$-almost every $r\in \R$. Then, we may assume that $r_1<r_2$.	
	
	Define two sequences $\mi_{1,n}$ and $\mi_{2,n}$ of bounded Radon measures on $X$ and $\R$, respectively, as follows. Let $(r_{1,n})$ be a strictly decreasing sequence which converges to $r_1-1$, and let $(r_{2,n})$ be a strictly increasing sequence which converges to $r_2$, and define 
	\[
	\mi_{1,n}\coloneqq [(\chi_{[r_{1,n},r_{2,n}]}\circ N) \psi^{-p'/p}]\cdot \mi_1
	\]
	for every $n\in\N$. In addition, let $(r'_{2,n})$ be a strictly increasing sequence which converges to $r_2+1$, and let $(r'_{1,n})$ be a strictly decreasing sequence which converges to $r_1$, and define 
	\[
	\mi_{2,n}\coloneqq \chi_{[r'_{1,n}, r'_{2,n}]}\cdot \mi_2.
	\]
	Then, the preceding considerations show that $\mi_{1,n}$ and $\mi_{2,n}$ are bounded inner regular measures for every $n\in\N$. In addition, $N_*(\mi_{1,n})$ and $\mi_{2,n}$ are compactly supported for every $n\in\N$. Now, take a positive $\Mf$-measurable function $f$ on $X$, so that $f\psi^{p'/p}$ is $\mi_{1,n}$-measurable for every $n\in\N$. Then, by monotone convergence,
	\[
	\begin{split}
		&\left( \int_\R \left( \int_{N(x)<r} f(x)\,\dd \mi_1(x)\right) ^q \,\dd \mi_2(r)\right)^{1/q}\\
			&\qquad=\lim_{n\to \infty} \left( \int_\R \left( \int_{N(x)<r} (f \psi^{p'/p})(x)\,\dd \mi_{1,n}(x)\right) ^q \,\dd \mi_{2}(r)\right)^{1/q}\\
			&\qquad=\lim_{n\to \infty}\lim_{m\to \infty} \left( \int_\R \left( \int_{N(x)<r} (f \psi^{p'/p})(x)\,\dd \mi_{1,n}(x)\right) ^q \,\dd \mi_{2,m}(r)\right)^{1/q}\\
			&\qquad\meg \liminf_{n\to \infty} C_B  \left( \int_X f^p \psi^{p'}\,\dd \mi_{1,n}\right)^{1/p}\\
			&\qquad\meg C_B  \left( \int_X f^p\,\dd (\psi\cdot\mi_1)\right)^{1/p}\\
			&\qquad\meg C_B\left( \int_X f^p\,\dd \mi_3\right)^{1/p}.
	\end{split}
	\]
	Thus,~\eqref{eq:1} holds for every positive $\Mf$-measurable function $f$ on $X$.

	\textsc{Step VII.} Assume that there is a finite constant $C$ such that~\eqref{eq:1} holds for every positive $\Mf$-measurable function $f$ on $X$. Observe that~\eqref{eq:1} holds with $\psi\cdot \mi_1$ in place of $\mi_3$ as one sees taking $f=0$ on a $\mi_1$-negligible Borel subset of $X$ where $\mi_3-\psi\cdot \mi_1$ is concentrated. Therefore, we may assume that $\mi_3=\psi\cdot \mi_1$. Take $r\in\R$. If $E$ is an $\Mf$-measurable subset of $N^{-1}((-\infty,r))$ such that $\mi_3(E)=0$, then applying~\eqref{eq:1} with $f=\chi_E$ we see that
	\[
	\begin{split}
		\mi_1(E)\mi_2([r,+\infty))^{1/q}&\meg \left( \int_\R \left( \int_{N(x)<r} f(x)\,\dd \mi_1(x)\right) ^q \,\dd \mi_2(r)\right)^{1/q}\\
		&\meg C\left(\int_X f^p\,\dd \mi_3\right)^{1/p}=0.
	\end{split}
	\]
	If $\mi_1(E) >0$, this implies that
	\[
	\mi_2([r,+\infty))=0, 
	\]
	so that
	\[
	\mi_2([r,+\infty))^{1/q} \left( \int_{N(x)< r} \psi(x)^{-p'/p}\,\dd \mi_1(x)\right)^{1/p'}=0.
	\]
	Therefore, if either $\mi_1(N^{-1}((-\infty,r)))=0$ or if $N^{-1}((-\infty,r))$ contains an $\Mf$-measurable subset $E$ with $\mi_1(E)>0$ and $\mi_3(E)=0$, then
	\[
	\mi_2([r,+\infty))^{1/q} \left( \int_{N(x)< r} \psi(x)^{-p'/p}\,\dd \mi_1(x)\right)^{1/p'}=0\meg C.
	\]
	Then, assume that   $\mi_1(N^{-1}((-\infty,r)))>0 $ and that $\mi_3(E)>0$ for every $\Mf$-measurable subset $E$ of $N^{-1}((-\infty,r))$ with $\mi_1(E)>0$. In other words, we assume that $\psi(x)>0$ for $\mi_1$-almost every $x\in N^{-1}((-\infty,r))$, and that $\mi_1(N^{-1}((-\infty,r)))>0$.
	 Since $\mi_1$ is semi-finite, also $\psi^{-p'/p}\cdot \mi_1$ is semi-finite (on $N^{-1}((-\infty,r))$), so that $(\psi^{-p'/p}\cdot \mi_1)(N^{-1}((-\infty,r)))$ is the least upper bound of the $(\psi^{-p'/p}\cdot \mi_1)(E)$, as $E$ runs through the set of $\Mf$-measurable subsets of $N^{-1}((-\infty,r))$ such that $0<(\psi^{-p'/p}\cdot \mi_1)(E)<\infty$. Then, fix such an $E$, and choose $f=\chi_E \psi ^{-p'/p}$ in~\eqref{eq:1}, so that
	\[
	\begin{split}
		\mi_2([r,+\infty))^{1/q}\int_{E}  \psi ^{-p'/p}\,\dd \mi_1 &\meg \left( \int_\R \left( \int_{N(x)<r} f(x)\,\dd \mi_1(x)\right) ^q \,\dd \mi_2(r)\right)^{1/q}\\
		&\meg C  \left( \int_E \psi^{1-p'}\,\dd \mi_1 \right)^{1/p}\\
		&=C  \left( \int_E \psi^{-p'/p}\,\dd \mi_1 \right)^{1/p}.
	\end{split}
	\]
	Therefore,
	\[
	\mi_2([r,+\infty))^{1/q} \left( \int_{E} \psi^{-p'/p}\,\dd \mi_1\right)^{1/p'}\meg C.
	\]
	The arbitrariness of $E$ then shows that
	\[
	\mi_2([r,+\infty))^{1/q} \left( \int_{N(x)< r} \psi(x)^{-p'/p}\,\dd \mi_1(x)\right)^{1/p'}\meg C.
	\] 
	By the arbitrariness of $r$, we have thus proved that $B\meg C$.
\end{proof}

\begin{lem}\label{lem:2}
	Let $f,g\colon \R\to [0,+\infty)$ be two functions, with $f$ decreasing and left-continuous, $g$ increasing and right-continuous, and $f g$ compactly supported. Then,
	\[
	\lim_{\eps\to 0^+} \sup_{r\in \R} f(r-\eps) g(r)=\sup_{r\in \R} f(r) g(r).
	\]
\end{lem}

\begin{proof}
	We may assume that $f,g\neq 0$, for otherwise the assertion is trivial.
	Take $a,b\in \R$ so that $f(b-\eps)>0=f(b+\eps)$ $g(a-\eps)=0<g(b+\eps)$ for every $\eps>0$: this is possible since$f,g\neq 0$ and $f g$ has compact support. In addition, the assertion is trivial if $b<a$, so that we may assume that $a\meg b$. Observe that, since $f(r-\eps)\to f(r)$ for every $r\in \R$, and for $\eps\to 0^+$, thanks to the left continuity of $f$, it is clear that
	\[
	\liminf_{\eps\to 0^+} \sup_{r\in \R} f(r-\eps) g(r)\Meg\sup_{r\in \R} f(r) g(r).
	\]
	Conversely, let $(\eps_k)$ be a sequence in $(0,+\infty)$ such that $\eps_k\to 0$ and 
	\[
	\limsup_{\eps\to 0^+} \sup_{r\in \R} f(r-\eps) g(r)=\lim_k \sup_{r\in \R} f(r-\eps_k) g(r).
	\]
	Take $\eta>0$. For every $k\in\N$, take $r_k\in [a,b+\eps_k]$ so that $\sup_{r\in \R} f(r-\eps_k) g(r)\meg f(r_k-\eps_k)g(r_k)+\eta$. Up to a subsequence, we may then assume that $(r_k)$ is a monotone sequence; let $\bar r\in [a,b]$ be its limit. If $(r_k)$ is increasing, then, $f(r_k-\eps_k)\to f(\bar r)$ since $f$ is left-continuous, and $g(r_k)\meg g(\bar r)$ since $g$ is increasing, so that
	\[
	\limsup_{\eps\to 0^+} \sup_{r\in \R} f(r-\eps) g(r)\meg \limsup_{k\to \infty} f(r_k-\eps_k) g(r_k)+\eta\meg   f(\bar r) g(\bar r)+\eta \meg \sup_{r\in \R} f(r) g(r)+\eta.
	\]
	If, otherwise, $(r_k)$ is decreasing, then $f(r_k-\eps_k)\meg f(\bar r-\eps_k)\to f(\bar r)$ since $f$ is decreasing and left-continuous, and $g(r_k)\to g(\bar r)$ since $g$ is right-continuous, so that
	\[
	\limsup_{\eps\to 0^+} \sup_{r\in \R} f(r-\eps) g(r)\meg \limsup_{k\to \infty} f(r_k-\eps_k) g(r_k)+\eta\meg f(\bar r) g(\bar r)+\eta \meg \sup_{r\in \R} f(r) g(r)+\eta.
	\]
	By the arbitrariness of $\eta$, this implies that 
	\[
	\limsup_{\eps\to 0^+} \sup_{r\in \R} f(r-\eps) g(r)  \meg \sup_{r\in \R} f(r) g(r),
	\]
	whence the result.
\end{proof}

\begin{proof}[Proof of Theorem~\ref{teo:1bis}]
	As in the proof of Thereom~\ref{teo:1}, we shall assume, for simplicity, $p,q\in (1,\infty)$, and leave to the reader the simple modifications needed to deal with the remaining cases.
	
	\textsc{Step I.} Assume first that: $B'$ is finite; $X=\R$; $\Mf$ is the Borel $\sigma$-algebra of $\R$; $N(r)=r$ for every $r\in \R$; $\mi_1$ and $\mi_2$ are bounded and compactly supported; $\mi_1=\mi_3$. Then, the assertion follows from Theorem~\ref{teo:1} if $\mi_1(\Set{r})\mi_2(\Set{r})=0$ for every $r\in \R$, since, in this case,	for every positive $\Mf$-measurable function $f$ on $\R$, one has $\int_{(-\infty,r)} f\,\dd \mi_1= \int_{(-\infty,r]} f\,\dd \mi_1$ for $\mi_2$-almost every $r\in\R$, and
	\[
	B'=\sup_{r\in \R} \mi_2([r,+\infty))^{1/q} \mi_1((-\infty,r))^{1/p'}
	\]
	since $\mi_2([\,\cdot\,,+\infty))$ is continuous at $r$ for every $r\in\R$ such that $\mi_1(\Set{r})>0$. In the general case, observe that, since both $\mi_1$ and $\mi_2$ have at most countably many atoms,  there is a sequence $(\eps_k)_{k\in\N}$ of elements of $(0,+\infty)$ such that $\eps_k\to 0$ and such that $\mi_1(\Set{r})\mi_2(\Set{r-\eps_k})=0$ for every $r\in\R$. Then, by Lemma~\ref{lem:2}
	\[
	B'=\lim_{k\to \infty} \sup_{r\in \R} \mi_2([r-\eps_k,+\infty))^{1/q} \mi_1((-\infty,r])^{1/p'},
	\]
	so that, for every positive $\Mf$-measurable function $f$ on $\R$,
	\[
	\begin{split}
		\left(\int_\R \left( \int_{(-\infty,r]} f\,\dd \mi_1 \right)^q\,\dd \mi_2(r)\right)^{1/q}&\meg\liminf_{k\to \infty} \left(\int_\R \left( \int_{(-\infty,r+\eps_k]} f\,\dd \mi_1 \right)^q\,\dd \mi_2(r)\right)^{1/q}\\
			&\meg C'\left(\int_\R f^p\,\dd \mi_1 \right)^{1/p}
	\end{split}
	\]
	by the right-continuity of the function $r\mapsto  \int_{(-\infty,r]} f\,\dd \mi_1 $, Fatou's lemma, and the previous remarks, with $C'=\min((p')^{1/p'} p^{1/q}, (q')^{1/p'} q^{1/q} ) B'$.
	Thus,~\eqref{eq:1bis} holds for every positive $\Mf$-measurable function $f$ on $\R$.
	
	\textsc{Step II.} Now, assume that $B'$ is finite, and let us prove that~\eqref{eq:1bis} holds for every positive $\Mf$-measurable function $f$ on $X$, with $C'=\min((p')^{1/p'} p^{1/q}, (q')^{1/p'} q^{1/q} ) B'$. The proof proceeds as in~\textsc{steps V} and \textsc{VI} of the proof of Theorem~\ref{teo:1}, with the following modification: in the proof of~\textsc{step VI}, one chooses $r_{2,n}=r_2$ if $r_2\in \R$ and $\mi_2(\Set{r_2})>0$ (in which case the finiteness of $B'$ implies that $\mi_1(N^{-1}((-\infty,r_2])))$ is finite), and, analogously, $r'_{1,n}=r_1$ if $r_1\in \R$ and $\mi_1(N^{-1}(\Set{r_1}))>0$ (in which case the finiteness of $B'$ implies that $\mi_2([r_1,+\infty))$ is finite).
	
	\textsc{Step III.} Finally, assume that there is a finite $C'$ for which~\eqref{eq:1bis} holds for every positive $\Mf$-measurable function $f$ on $X$, and let us prove that $B'\meg C'$. The proof proceeds as in~\textsc{step VII} of the proof of Theorem~\ref{teo:1}, simply replacing $(-\infty,r)$ with $(-\infty,r]$ (and $N(x)<r$ with $N(x)\meg r$) everywhere.
\end{proof}

\end{document}